\documentclass[11pt]{amsart}
\title{The Loop Murnaghan-Nakayama Rule}
\author{Dustin Ross}
\address{Dustin Ross, Colorado State University, Department of Mathematics, Fort Collins, CO 80523-1874, USA}
\email{ross@math.colostate.edu}

\usepackage {color,graphicx,psfrag,verbatim,amssymb,amscd,enumerate,subfigure,ytableau}
\usepackage{texdraw,appendix}

\newcommand{\N}{\mathbb{N}}
\newcommand{\Z}{\mathbb{Z}}

\newcommand{\C}{\mathbb{C}}

\newtheorem{dummy}{}[section]
\newtheorem{lemma}[dummy]{Lemma}

\newtheorem*{theorem1}{Theorem 1}
\newtheorem*{theorem2}{Theorem 2}
\newtheorem*{theorem1'}{Theorem 1'}
\newtheorem*{theorem2'}{Theorem 2'}

\theoremstyle{definition}

\newtheorem{definition}[dummy]{Definition}
\newtheorem{example}[dummy]{Example}

\newtheorem{remark}[dummy]{Remark}

\usepackage{graphicx}

2
\begin{document}


\begin{abstract}
We give a combinatorial proof of a natural generalization of the Murnaghan-Nakayama rule to loop Schur functions.  We also introduce shifted loop Schur functions and prove that they satisfy a similar relation. 
\end{abstract}

\maketitle


\section{Introduction}\label{sec:intro}

\subsection{Statement of Results}

The Schur functions $s_\lambda(\mathbf{x})$ are a special class of power series defined in infinitely many variables $\mathbf{x}=(x_1,x_2,...)$ and indexed by partitions $\lambda$ (we refer the reader to Section \ref{sec:definitions} for a precise definition).  Schur functions are classically known to form an orthonormal, integral basis of the ring of symmetric functions and they have proven ubiquitous in many areas of mathematics.  

Another (rational) basis for the ring of symmetric functions is given by products of the power-sum functions $p_k(\mathbf{x})$.  The classical Murnaghan-Nakayama rule provides a simple way to write the symmetric function $p_ks_\lambda$ in the Schur basis:
\begin{equation}\label{classicalmn}
p_ks_\lambda=\sum_\sigma(-1)^{ht(\sigma\setminus\lambda)}s_\sigma
\end{equation}
where the sum is over all ways of adding a length $k$ \textit{border strip} to $\lambda$ and $ht$ is the height (i.e. the number of rows) of the border strip, minus $1$.

Loop Schur functions naturally generalize the combinatorial definition of Schur functions and have previously been studied in the context of loop symmetric functions (\cite{lp:tpilgi}).  Given a positive integer $n$, the loop Schur functions $s_{\lambda}[n]$ are power series in infinitely many variables $\{x_{i,j}:i\in \Z_n,j\in\N\}$ and indexed by partitions $\lambda$.  There is also a notion of \textit{loop power-sum functions} $p_k[n]$.  In Section \ref{sec:thm1} we provide a combinatorial proof for the natural generalization of the Murnaghan-Nakayama rule.

\begin{theorem1}\label{main}
\[
p_k[n]s_\lambda[n]=\sum_\sigma (-1)^{ht(\sigma\setminus\lambda)}s_\sigma[n]
\]
where the sum is over all ways of adding length $kn$ border strips to $\lambda$.
\end{theorem1}

By forgetting the index $i\in\Z_n$, Theorem 1 specializes to the classical Murnaghan-Nakayama rule and, to the best of our knowledge, our proof provides a \textit{new} combinatorial proof of the classical result.

For any $0\leq l< n$, we introduce in Section \ref{sec:definitions} the \textit{$l$-shifted loop Schur functions} $s_\lambda^l[n]$, a close variant of the loop Schur functions (in particular, $s_\lambda^0[n]=s_\lambda[n]$).  We prove the following identity in Section \ref{sec:theorem2}.

\begin{theorem2}\label{vanish}
For $l\neq 0$,
\[
0=\sum_\sigma (-1)^{ht(\sigma\setminus\lambda)}s^l_\sigma[n]
\]
where the sum is over all ways of adding length $kn$ border strips to $\lambda$.
\end{theorem2}

\subsection{Context and Motivation}

The motivation which led us to the study of loop Schur functions lies in geometry, particularly in the study of curves in complex threefolds.  Gromov-Witten (GW) and Donaldson-Thomas (DT) theory define invariants of complex threefolds which virtually count curves with prescribed incidence conditions (c.f. \cite{pt:wtcc} for an introductory survey of these and other curve counting theories).  It was conjecture in \cite{mnop:gwdt1,mnop:gwdt2} that GW and DT theory coincide for smooth threefolds (i.e. there is a change of variables which equates the corresponding generating functions).  The GW/DT correspondence was proven in \cite{moop:gwdtc} for the case of toric threefolds.  The proof in \cite{moop:gwdtc} used the presence of torus actions and Atiyah-Bott localization to reduce the correspondence to the level of the equivariant ``topological vertex'' - a formal generating function associated to each fixed point of the toric threefold along with an algorithm to recover the full GW or DT theory.

GW and DT theory have recently been defined for three dimensional orbifolds, i.e. spaces which are locally modeled by finite quotients of $\C^3$ (\cite{cr:ogwt,agv:gwtodms,bcy:otv}).  Moreover, the topological vertex algorithm has been generalized to three dimensional toric orbifolds in both GW theory (\cite{r:lgoa}) and DT theory (\cite{bcy:otv}).  In GW theory, the orbifold vertex is a generating function of abelian Hodge integrals, whereas in the DT case it is a generating function of colored plane partitions.  Naturally, one would expect a correspondence of these theories generalizing the correspondence in the smooth case.  The loop Murnaghan-Nakayama rule arose in the study of such a correspondence for the orbifold vertex associated to the threefold $A_{n}$ singularity: $\left[\C^3/\Z_{n+1}\right]$ where $\Z_{n+1}$ acts with weights $(1,-1,0)$.

In particular, the results of \cite{er:cpgisf} and \cite{bcy:otv} show that a certain specialization of the loop Schur functions are closely related to the Donaldson-Thomas vertex for the $A_n$ singularity (more specifially, a `one-leg' specialization of the DT vertex).  The corresponding GW orbifold vertex is determined by a certain set of bilinear relations developed in \cite{z:ologwv}.  In \cite{r:ggmv}, these relations are reduced to identities in the DT vertex which are specializations of Theorems 1 and 2.  Therefore, Theorems 1 and 2 conclude the proof of the orbifold GW/DT correspondence for this particular orbifold vertex.  From the one-leg vertex correspondence, it follows that the GW and DT theories agree for any local orbifold line with cyclic isotropy (\cite{r:ggmv}).  Generalizing this correspondence to the full three leg vertex is currently under investigation by the author and Zong and will ultimately require several new combinatorial identities involving loop Schur functions.

Loop Schur functions have independently been studied in \cite{lp:tpilgi}.  Lam and Pylyavskyy proved a number of interesting properties concerning the loop Schur functions.  In particular, they showed that loop Schur functions belong to the ring of loop symmetric functions, i.e. they are invariants of a birational symmetric group action on an appropriate polynomial ring.  In fact, Lam and Pylyavskyy have announced Theorem 1 in \cite{l:lsfafmp}, though their proof (which presumably uses very different techniques than ours) has not yet appeared in the literature.  For a very nice survey of these results including applications to networks on surfaces, total positivity, crystal graphs, and discrete integrable systems, see \cite{l:lsfafmp}.



It should also be mentioned that (some relative of) the loop Schur functions have been studied in \cite{er:cpgisf} and \cite{n:hfgt}.

\subsection{Outline of the Proofs}

Theorem 1 is proven as a corollary of the following identity
\begin{theorem1'}
For any $N\geq kn+l(\lambda)$, 
\[
p_{k,N}[n] x^\delta s_{\lambda,N}[n]=\sum(-1)^{ht(\sigma\setminus\lambda)}x^\delta s_{\sigma,N}[n]
\]
where
\[
x^\delta:=(x_{-1,1}\cdot...\cdot x_{-N,1})(x_{-2,2}\cdot...\cdot x_{-N,2})...(x_{-N,N}),
\]
$p_{k,N}[n]$, $s_{\lambda,N}[n]$ are defined by specializing $x_{i,j}=0$ if $j>N$, and the sum is over all ways of adding length $kn$ border strips to $\lambda$.

\end{theorem1'}

To prove Theorem 1', we begin in Section \ref{sec:inv1} by interpreting the product $x^\delta s_{\lambda,N}[n]$ combinatorially in a way which will be convenient for later arguments - the key tool is a sign-reversing involution which has previously been defined in \cite{ckr:cpeccdsf}.  In Section \ref{sec:master}, we define a master generating function $F_{\lambda,N}[n]$ for certain combinatorial gadgets closely related to those discussed in Section \ref{sec:inv1}.  In Sections \ref{sec:inv2} and \ref{sec:inv3} we define sign-reversing involutions on the terms in $F_{\lambda,N}[n]$ with the property that the sum of the weights of the fixed terms can be identified with the left and right-hand sides, respectively, of Theorem 1'.  This proves that both sides are equal to $F_{\lambda,N}[n]$, thus proving the theorem.

Theorem 2 follows quickly in Section \ref{sec:theorem2} using similar techniques.

\subsection{Acknowledgements} 
A great deal of gratitude is owed to my advisor, Renzo Cavalieri, for his expert guidance.  I am also grateful to Thomas Lam for bringing to my attention his work with Pavlo Pylyavskyy.  For invaluable conversations, suggestions, and encouragement, I am also indebted to R. Croke, T. Gern, P. Johnson M. Konvalinka, S. Kovacs, S. Mkrtchyan, K. Monks, E. Nelson, T. Penttila, C. Peterson, L. Serrano, and C. Strickland.

\section{Definitions and Notation}\label{sec:definitions}

In this section we make precise the objects which appeared in the statements of Theorems 1 and 2.  Before defining loop Schur functions, we begin by briefly recalling the classical Schur functions (see e.g. \cite{m:sfhp}).  Though originally defined as quotients of antisymmetric functions, Schur functions can be defined combinatorially as generating functions of semi-standard Young tableaux as we now describe.

To a partition $\lambda$ we can associate a Young diagram (which we also call $\lambda$), a northwest justified collection of boxes where the rows encode the sizes of the parts of $\lambda$.  For example, if $\lambda$ is the partition $(4,3,3,2)$, the associated Young diagram is:
\[
\lambda=\ydiagram{4,3,3,2}
\]

A \textit{tableau} of $\lambda$ is an assignment of positive integers to the boxes of $\lambda$.  A \textit{semi-standard Young tableau} (SSYT) of $\lambda$ is a numbering of the boxes so that numbers are weakly increasing left to right and strictly increasing top to bottom.  For each $\square\in\lambda$, we define the \textit{weight} $w(\square,T)$ to be the number appearing in that square.  To each tableau $T\in SSYT(\lambda)$ we can associate a monomial
\[
x^T:=\prod_{\square\in\lambda} x_{w(\square,T)}.
\]
For example, to the SSYT
\[
T= \begin{ytableau}
1 & 1 & 2 & 4\\
2 & 3 & 3\\
4 & 4 & 6\\
7 & 7
\end{ytableau}
\]
we associate the monomial $x^T=x_1^2x_2^2x_3^2x_4^3x_6x_7^2$.

The Schur functions can be defined by the rule
\[
s_\lambda:=\sum_{T\in SSYT(\lambda)}x^T.
\]

It is not obvious, but this definition of Schur functions coincides with the classical definition (c.f. \cite{m:sfhp} or \cite{ckr:cpeccdsf} for a combinatorial proof).

The power-sum functions are defined as
\[
p_k:=\sum_i x_i^k.
\]

The sum in the classical Murnaghan-Nakayama rule \eqref{classicalmn} is over all Young diagrams $\sigma\supset\lambda$ such that the complement is connected, contains $k$ boxes, and contains no $2\times 2$ square.  We say that $\sigma$ is obtained from $\lambda$ by adding a length $k$ \textit{border strip} and $ht(\sigma\setminus\lambda)$ is the number of rows the border strip occupies, minus $1$.

\subsection{Loop Schur Functions}\label{sec:csf} 

In the current paper, we study loop Schur functions which we now define.  For a positive integer $n$ and partition $\lambda$, the \textit{colored} Young diagram $(\lambda,n)$ is obtained by coloring the boxes of the Young diagram by their \textit{content} modulo $n$.  In other words if $\square$ is in the $i$th row and the $j$th column (row and column indexing begins with $1$), we color it $c(\square):=j-i \mod n$.  For example, if $\lambda=(4,3,3,2)$ and $n=3$, the colored Young diagram is given by
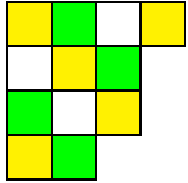
\begin{center}
\begin{ytableau}
*(yellow) & *(green) & *(white) & *(yellow)\\
*(white) & *(yellow) & *(green)\\
*(green) & *(white) & *(yellow)\\
*(yellow)& *(green)
\end{ytableau}
\end{center}
with
\[
0\leftrightarrow\begin{ytableau}*(yellow)\end{ytableau}, \hspace{1cm} 1\leftrightarrow\begin{ytableau}*(green)\end{ytableau}, \hspace{.5cm}\text{and}\hspace{.5cm} 2\leftrightarrow\begin{ytableau}*(white)\end{ytableau}.
\]

We let $\lambda[i]$ denote the collection of boxes with color $i$.  To each semi-standard Young tableau $T\in SSYT(\lambda,n)$, we associate a monomial in $n$ infinite sets of variables $\{x_{i,j}:i\in\Z_n,j\in\N\}$:
\begin{equation}\label{mon}
x^T:=\prod_{i=0}^{n-1}\prod_{\square\in\lambda[i]}x_{i,w(\square,T)}.
\end{equation}
For example, to the SSYT
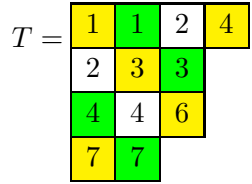
\[
T=\begin{ytableau}
*(yellow) 1& *(green) 1& *(white) 2& *(yellow) 4\\
*(white) 2& *(yellow) 3& *(green) 3\\
*(green) 4& *(white) 4& *(yellow) 6\\
*(yellow) 7 & *(green) 7
\end{ytableau}
\]
we associate the monomial
\[
x^T=x_{0,1}x_{0,3}x_{0,4}x_{0,6}x_{0,7}x_{1,1}x_{1,3}x_{1,4}x_{1,7}x_{2,2}^2x_{2,4}.
\]
\begin{definition}
The \textit{loop Schur function} associated to $(\lambda,n)$ is defined by
\[
s_{\lambda}[n]:=\sum_{T\in SSYT(\lambda,n)}x^T.
\]
\end{definition}

Power-sum functions also naturally generalize to the colored setting.

\begin{definition}
The \textit{loop power-sum functions} are defined by
\[
p_k[n]:=\sum_j\left( \prod_{i=0}^{n-1}x_{i,j} \right)^k.
\]
\end{definition}

\begin{remark}\label{spec1}
By definition we have the following specializations:
\[
p_k[n]|_{(x_{i,j}=x_j)}=p_{kn} \text{ and }  s_{\lambda}[n]|_{(x_{i,j}=x_j)}=s_\lambda.
\]
It follows immediately that Theorem 1 specializes to the classical identity \eqref{classicalmn} by forgetting the index $i$.
\end{remark}

\subsection{Shifted Loop Schur Functions}\label{sec:scsf}

To define the $l$-shifted loop Schur functions appearing in Theorem 2, we define the shifted weight 
\[
w^l(\square,T):=w(\square,T)+\frac{l\cdot c(\square)}{n}
\] 
and the corresponding monomial
\begin{equation}\label{shiftmon}
x^{T,l}:=\prod_{i=0}^{n-1}\prod_{\square\in\lambda[i]}x_{i,w^l(\square,T)}
\end{equation}
where the variables appearing in the monomial now belong to the set $\{x_{i,j}:i\in\Z_n,j\in\frac{1}{n}\Z\}$.

\begin{definition} The \textit{$l$-shifted loop Schur function} associated to $(\lambda,n)$ is defined by
\[
s_\lambda^l[n]:=\sum_{T\in SSYT(\lambda,n)}x^{T,l}.
\]
\end{definition}

\begin{remark}
By definition, $s_\lambda^0[n]=s_\lambda[n]$.
\end{remark}

\section{Proof of Theorem 1}\label{sec:thm1}

\subsection{Involutions: Round One}\label{sec:inv1}

In this section we give a combinatorial description of the product $x^\delta s_{\lambda,N}[n]$ which will prove useful in later arguments.  For a given Young diagram $\lambda$, and positive integers $n$ and $N>l(\lambda)$, define $\hat\lambda$ to be the diagram obtained by adding adding a staircase of size $N$ to the left of $\lambda$.  In other words, we add $N-i+1$ boxes to the left of the $i$th row of $\lambda$ (if $i>l(\lambda)$, the right edge of the new boxes should be justified with the right edge of the new boxes in the rows above it).  As before, the diagram is colored by content modulo $n$.  Consider pairs $(T,\tau)$ where
\begin{enumerate}[(i)]
\item $T$ is a tableau (not necessarily semi-standard) of $\hat\lambda$, and
\item $\tau=(\tau_1,...,\tau_N)$ is a labeling of the $N$ rows of $\hat\lambda $ with the numbers $1,...,N$ (considered as a permutation $\begin{pmatrix} 1 & 2 & ... & N \\ \tau_1 & \tau_2 & ... & \tau_N\\ \end{pmatrix}\in S_N$). 
\end{enumerate}

Let $\mathcal{T}_{\lambda,n,N}$ be the set of such pairs $(T,\tau)$ which satisfy the following conditions:
\begin{enumerate}[(i)]
\item $T$ only contains the numbers $1,...,N$.
\item The rows of $T$ are weakly increasing.
\item The leftmost entry in the $j$th row is at least $\tau_j$.
\end{enumerate}

\begin{remark}
When confusion does not arise, we omit the subscripts and write $\mathcal{T}=\mathcal{T}_{\lambda,n,N}$.
\end{remark}

\begin{example}\label{exi1}
For $\lambda=(2,1)$, $n=3$, and $N=5$, we give two examples of elements in $\mathcal{T}$.
\[
\begin{ytableau}
*(green)2 	&*(white)2	&*(yellow)3	&*(green)4	&*(white)4	&*(yellow)4	&*(green)5	& \none 	&\none[2] 	\\
\none 		&*(green)5	&*(white)5	&*(yellow)5	&*(green)5	&*(white)5	& \none		& \none		&\none[5]	\\		
\none		& \none		&*(green)4	&*(white)4	&*(yellow)5	& \none		& \none		& \none		&\none[4]	\\
\none		& \none		& \none		&*(green)2	&*(white)2	& \none		& \none		& \none		&\none[1]	\\
\none		& \none		& \none		& \none		&*(green)3	& \none		& \none		& \none		&\none[3]	\\	
\end{ytableau}
\hspace{2cm}
\begin{ytableau}
*(green)2 	&*(white)2	&*(yellow)3	&*(green)4	&*(white)4	&*(yellow)4	&*(green)5	& \none 	&\none[2] 	\\
\none 		&*(green)4	&*(white)4	&*(yellow)5	&*(green)5	&*(white)5	& \none		& \none		&\none[4]	\\		
\none		& \none		&*(green)5	&*(white)5	&*(yellow)5	& \none		& \none		& \none		&\none[5]	\\
\none		& \none		& \none		&*(green)2	&*(white)2	& \none		& \none		& \none		&\none[1]	\\
\none		& \none		& \none		& \none		&*(green)3	& \none		& \none		& \none		&\none[3]	\\	
\end{ytableau}
\]
\end{example}

As in (\ref{mon}), we can associate to each $T$ a monomial $x^T$.  Let $(-1)^{\tau}$ denote the sign of the permutation $\tau$.  We have the following identity.

\begin{lemma}\label{lem1}
\[
x^\delta s_{\lambda,N}[n]=\sum_{(T,\tau)\in\mathcal{T}}(-1)^{\tau}x^T.
\]
\end{lemma}

\begin{proof}
We consider a sign reversing involution which cancels pairs of terms in the sum.  We then identify the sum of the fixed terms as $x^\delta s_{\lambda,N}[n]$.  The involution we use is defined in \cite{ckr:cpeccdsf}, the setting here is only slightly different.  We include the details for completeness.

The involution $I_1$ is defined on a pair $(T,\tau)$ as follows:
\begin{enumerate}[(I)]
\item Look for the rightmost and then highest vertical domino $\tiny\ydiagram{1,1}$ such that the upper entry is at least the lower entry.
\item Swap every box to the left of the upper box in (I) with the box directly to its southeast.
\item Swap the elements of $\tau$ which index these two rows.
\end{enumerate}
Define $I_1(T,\tau)$ to be the new tableau and permutation obtained through this process.  We will often abuse notation and write $I_1(T,\tau)=(I_1(T),I_1(\tau))$ to reference the action on the tableau or the permutation alone.  See Example \ref{exi1} above for two elements of $\mathcal{T}$ which are interchanged by $I_1$.

First of all, $I_1(T,\tau)$ is an involution because the location of the domino in step (I) is preserved under the action.  It is easy to see that $x^T=x^{I_1(T)}$ since the involution moves entries along diagonals on which the colors are constant.  It is also easy to see that $(-1)^{\tau}=-(-1)^{I_1(\tau)}$ whenever $(T,\tau)$ is not fixed by $I_1$ because switching two elements of the labeling $\tau$ corresponds to multiplying the corresponding permutation by a transposition.  Therefore, we conclude that 
\[
\sum_{\mathcal{T}}(-1)^{\tau}x^T=\sum_{\mathcal{T}^{I_1}}(-1)^{\tau}x^T.
\]
where $\mathcal{T}^{I_1}$ is the set of elements in $\mathcal{T}$ which are fixed by $I_1$.

It is left to analyze $\mathcal{T}^{I_1}$.  If $(T,\tau)$ is fixed by $I_1$, then $T$ must be a column-strict tableau.  In particular, the column immediately to the left of $\lambda\subset\hat\lambda$ should read $1,...,N$ top to bottom.  In particular, this implies that the entries of $\hat\lambda\setminus\lambda$ must be $1$ in the first row, $2$ in the second row, etc. and $\tau$ is forced to be the identity.  The constraint imposed on the entries of $\lambda$ are simply that they form a semi-standard tableau.  The entries in $\hat\lambda\setminus\lambda$ contribute $x^\delta$ to each monomial $x^T$ and the sum over all semi-standard tableaux of $\lambda$ contributes $s_{\lambda,N}[n]$.

\end{proof}

\subsection{Master Generating Function}\label{sec:master}

In this section we define a master generating function $F_{\lambda,N}[n]$ which is shown in subsequent sections to equal both the left and right-hand sides of the identity in Theorem 1'.  To that end, we fix a partition $\lambda=(\lambda_1,...,\lambda_l)$, positive integers $n$ and $k$, and a positive integer $N$ satisfying $N\geq kn+l$.  For any $i\in\{1,...,N\}$, let $\hat\lambda_i$ be the diagram obtained by adding $kn$ boxes to the right of the $i$th row of $\hat\lambda$.  The combinatorial objects we want to consider are pairs $(T,\tau)$ where 
\begin{enumerate}[(i)]
\item $T$ is a tableau of the diagram $\hat\lambda_i$ for some $i$, and
\item $\tau=(\tau_1,...,\tau_N)$ is a labeling of the $N$ rows of $\hat\lambda_i$ with the numbers $1,...,N$ (considered as a permutation in $S_N$).   
\end{enumerate}
Let $\mathcal{S}_{\lambda,n,k,N}$ be the set of such tableaux which satisfy the same three conditions (i) - (iii) required of the set $\mathcal{T}$ in Section \ref{sec:inv1}.  

\begin{example}\label{exi2}
For $\lambda=(2,1)$, $n=3$, $k=1$, and $N=5$, we give two examples of elements in $\mathcal{S}$.
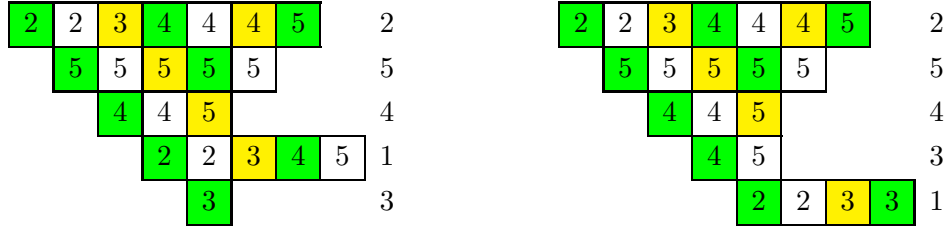
\[
\begin{ytableau}
*(green)2 	&*(white)2	&*(yellow)3	&*(green)4	&*(white)4	&*(yellow)4	&*(green)5	& \none 	&\none[2] 	\\
\none 		&*(green)5	&*(white)5	&*(yellow)5	&*(green)5	&*(white)5	& \none		& \none		&\none[5]	\\		
\none		& \none		&*(green)4	&*(white)4	&*(yellow)5	& \none		& \none		& \none		&\none[4]	\\
\none		& \none		& \none		&*(green)2	&*(white)2	&*(yellow)3	&*(green)4	&*(white)5	&\none[1]	\\
\none		& \none		& \none		& \none		&*(green)3	& \none		& \none		& \none		&\none[3]	\\	
\end{ytableau}
\hspace{2cm}
\begin{ytableau}
*(green)2 	&*(white)2	&*(yellow)3	&*(green)4	&*(white)4	&*(yellow)4	&*(green)5	& \none 	&\none[2] 	\\
\none 		&*(green)5	&*(white)5	&*(yellow)5	&*(green)5	&*(white)5	& \none		& \none		&\none[5]	\\		
\none		& \none		&*(green)4	&*(white)4	&*(yellow)5	& \none		& \none		& \none		&\none[4]	\\
\none		& \none		& \none		&*(green)4	&*(white)5	&\none		&\none		&\none		&\none[3]	\\
\none		& \none		& \none		& \none		&*(green)2	&*(white)2	&*(yellow)3	&*(green)3	&\none[1]	\\	
\end{ytableau}
\]
\end{example}

To each $(T,\tau)$, we assign a monomial $x^T$ as before.  We define the generating function $F_{\lambda,N}[n]$ by
\begin{equation}\label{master}
F_{\lambda,N}[n]:=\sum_{(T,\tau)\in\mathcal{S}}(-1)^{\tau}x^T.
\end{equation}

\subsection{Involutions: Round Two}\label{sec:inv2}

\begin{lemma}\label{lem2}
\[
F_{\lambda,N}[n]=p_{k,N}[n] x^\delta s_{\lambda,N}[n]
\]
\end{lemma}

\begin{proof}
We define an involution on the terms of $F_{\lambda,N}[n]$ which cancels terms in pairs.  The remaining terms are seen to coincide with the left-hand side of Theorem 1'.  We define the involution $I_2$ on sets of pairs $(T,\tau)$ as follows.  

If $T$ is a tableau of $\hat\lambda_i$ and the $kn$th entry of row $i$ is $\tau_i$, then define $I_2(T,\tau)=(T,\tau)$.  Otherwise, the $kn$th entry of row $i$ is $l$ with $l>\tau_i$ because of conditions (ii) and (iii) in Section \ref{sec:inv1}.  Then $I_2(T,\tau)$ is defined by the following process:
\begin{enumerate}[(I)]
\item Remove the first $kn$ boxes (along with their labels) of the $i$th row, and shift the remaining boxes in that row to the left by $kn$ units.
\item Interchange $\tau_i$ and $\tau_j$ where $j$ is the row with $\tau_j=l$.
\item Slide the boxes in row $j$ to the right by $kn$ units and reinsert the $kn$ boxes (along with their labels) in row $j$.
\end{enumerate}

Notice that when $I_2$ does not fix an element, it sends a tableau of $\hat\lambda_i$ to a tableau of $\hat\lambda_j$ with $j\neq i$.  We will again abuse notation and write $I_2(T,\tau)=(I_2(T),I_2(\tau))$.  See Example \ref{exi2} for an illustration of two elements of $\mathcal{S}$ which are interchanged by $I_2$.

It is easy to see that $I_2$ is an involution, $x^T=x^{I_2(T)}$,and if $(T,\tau)$ is not a fixed point of $I_2$, then $(-1)^{\tau}=-(-1)^{I_2(\tau)}$.  Therefore, the terms which are not fixed cancel in pairs in the sum (\ref{master}).

By definition, the terms which are fixed correspond to those where the $kn$ leftmost boxes in $\hat\lambda_i$ all contain the number $\tau_i$.  If we set $\underline{N}:=\{1,...,N\}$, then we obtain a bijection between the sets $\mathcal{S}^{I_2}$ and $\mathcal{T}\times\underline{N}$ by mapping $(T,\tau)$ to $(T',\tau,i)$ where $T'$ is obtained by removing the leftmost $kn$ boxes from row $i$ and sliding the remaining boxes to the left.  Moreover, this bijection preserves $(-1)^{\tau}$ and the weights are related by the equation $x^T=\left(\prod_{j=0}^{n-1}x_{j,i}\right)^kx^{T'}$.  We have
\begin{align*}
F_{\lambda,N}[n]&=\sum_{(T,\tau)\in\mathcal{S}}(-1)^{\tau}x^T\\
&=\sum_{(T,\tau)\in\mathcal{S}^{I_2}}(-1)^{\tau}x^T\\
&=\sum_{(T',\tau,i)\in\mathcal{T}\times\underline{N}}(-1)^{\tau}\left(\prod_{j=0}^{n-1}x_{j,i}\right)^kx^{T'}\\
&=\left( \sum_{i=1}^N\left(\prod_{j=0}^{n-1}x_{j,i}\right)^k \right)\left( \sum_{(T',\tau)\in \mathcal{T}}(-1)^{\tau}x^{T'} \right)\\
&=p_{k,N}[n] x^\delta s_{\lambda,N}[n]
\end{align*}
where the last equality follows from Lemma \ref{lem1} and the definition of the loop power-sum functions.
\end{proof}

\subsection{Involutions: Round Three}\label{sec:inv3}

\begin{lemma}\label{lem3}
\[
F_{\lambda,N}[n]=\sum(-1)^{ht(\sigma\setminus\lambda)}x^\delta s_{\sigma,N}[n]
\]
where the sum is over all ways of adding a length $kn$ border strip to $\lambda$.
\end{lemma}

\begin{proof}
We define a different involution on $\mathcal{S}$ which cancels terms in the sum $F_{\lambda,N}[n]$ in pairs.  The sum of the weights of the remaining terms is then seen to coincide with $\sum(-1)^{ht(\sigma\setminus\lambda)}x^\delta s_{\sigma,N}[n]$.  The involution $I_3$ is defined as follows.

First, if $(T,\tau)$ is a tableau on $\hat\lambda_i$ and two rows of $\hat\lambda_i$ have the same number of boxes, then one of those rows must be $i$, call the other one $j$ (it is not hard to see that at most two rows can have equal length).  Define $I_3(T,\tau)=(T^*,\tau^*)$ where $T^*$ is obtained by swapping the entries of rows $i$ and $j$ and $\tau^*$ is obtained by swapping $\tau_i$ and $\tau_j$.

\begin{example}
$I_3$ interchanges the following elements of $\mathcal{S}$.
\[
\begin{ytableau}
*(green)2 	&*(white)2	&*(yellow)3	&*(green)4	&*(white)4	&*(yellow)4	&*(green)5	& \none 	&\none[2] 	\\
\none 		&*(green)5	&*(white)5	&*(yellow)5	&*(green)5	&*(white)5	& \none		& \none		&\none[5]	\\		
\none		& \none		&*(green)4	&*(white)4	&*(yellow)5	& \none		& \none		& \none		&\none[4]	\\
\none		& \none		& \none		&*(green)2	&*(white)2	&*(yellow)3	&*(green)4	&*(white)5	&\none[1]	\\
\none		& \none		& \none		& \none		&*(green)3	& \none		& \none		& \none		&\none[3]	\\	
\end{ytableau}
\longleftrightarrow\hspace{.5cm}
\begin{ytableau}
*(green)2 	&*(white)2	&*(yellow)3	&*(green)4	&*(white)4	&*(yellow)4	&*(green)5	& \none 	&\none[2] 	\\
\none 		&*(green)2	&*(white)2	&*(yellow)3	&*(green)4	&*(white)5	& \none		& \none		&\none[1]	\\		
\none		& \none		&*(green)4	&*(white)4	&*(yellow)5	& \none		& \none		& \none		&\none[4]	\\
\none		& \none		& \none		&*(green)5	&*(white)5	&*(yellow)5	&*(green)5	&*(white)5	&\none[5]	\\
\none		& \none		& \none		& \none		&*(green)3	& \none		& \none		& \none		&\none[3]	\\	
\end{ytableau}
\]
\end{example}

If all rows of $\hat\lambda_i$ have distinct size, then $I_3(T,\tau)$ is obtained as follows:
\begin{enumerate}[(I)]
\item Slide the $i$th row of $\hat\lambda_i$ northwest until the length of the rows are strictly decreasing, slide $\tau_i$ upward with the row, call this new tableau $(T',\tau')$. 
\item Apply the involution $I_1$ from the proof of Lemma \ref{lem1} to $(T',\tau')$.
\item Reverse step (I).
\end{enumerate}

\begin{remark}
The important thing to notice is that the new diagram obtained in Step I can be identified with $\hat\sigma$ for some $\sigma$ which is obtained from $\lambda$ by adding a length $kn$ border strip.
\end{remark}

\begin{example}
This example illustrates the involution $I_3$.  The first diagram is $(T,\tau)$, the second is $(T',\tau')$, the third is $I_1(T',\tau')$, and the fourth is $I_3(T,\tau)$.
\begin{align*}
\begin{ytableau}
*(green)2 	&*(white)2	&*(yellow)3	&*(green)4	&*(white)4	&*(yellow)4	&*(green)4	& \none 	&\none[2] 	\\
\none 		&*(green)2	&*(white)2	&*(yellow)3	&*(green)4	&*(white)4	& \none		& \none		&\none[1]	\\		
\none		& \none		&*(green)4	&*(white)4	&*(yellow)5	&*(green)5	&*(white)5	&*(yellow)5	&\none[4]	\\
\none		& \none		& \none		&*(green)5	&*(white)5	& \none		& \none		& \none		&\none[5]	\\
\none		& \none		& \none		& \none		&*(green)3	& \none		& \none		& \none		&\none[3]	\\	
\end{ytableau}
\longrightarrow\hspace{.5cm}
&\begin{ytableau}
*(green)2 	&*(white)2	&*(yellow)3	&*(green)4	&*(white)4	&*(yellow)4	&*(green)4	& \none 	&\none[2] 	\\
\none		&*(green)4	&*(white)4	&*(yellow)5	&*(green)5	&*(white)5	&*(yellow)5	& \none		&\none[4]	\\
\none 		& \none		&*(green)2	&*(white)2	&*(yellow)3	&*(green)4	&*(white)4	& \none		&\none[1]	\\		
\none		& \none		& \none		&*(green)5	&*(white)5	& \none		& \none		& \none		&\none[5]	\\
\none		& \none		& \none		& \none		&*(green)3	& \none		& \none		& \none		&\none[3]	\\
\end{ytableau}\\
\longrightarrow \hspace{.5cm}
\begin{ytableau}
*(green)2 	&*(white)2	&*(yellow)3	&*(green)4	&*(white)4	&*(yellow)4	&*(green)4	& \none 	&\none[2] 	\\
\none		&*(green)2	&*(white)2	&*(yellow)3	&*(green)4	&*(white)4	&*(yellow)5	& \none		&\none[1]	\\
\none 		& \none		&*(green)4	&*(white)4	&*(yellow)5	&*(green)5	&*(white)5	& \none		&\none[4]	\\		
\none		& \none		& \none		&*(green)5	&*(white)5	& \none		& \none		& \none		&\none[5]	\\
\none		& \none		& \none		& \none		&*(green)3	& \none		& \none		& \none		&\none[3]	\\
\end{ytableau}
\longrightarrow\hspace{.5cm}
&\begin{ytableau}
*(green)2 	&*(white)2	&*(yellow)3	&*(green)4	&*(white)4	&*(yellow)4	&*(green)4	& \none 	&\none[2] 	\\
\none 		&*(green)4	&*(white)4	&*(yellow)5	&*(green)5	&*(white)5	& \none		& \none		&\none[1]	\\		
\none		& \none		&*(green)2	&*(white)2	&*(yellow)3	&*(green)4	&*(white)4	&*(yellow)5	&\none[4]	\\
\none		& \none		& \none		&*(green)5	&*(white)5	& \none		& \none		& \none		&\none[5]	\\
\none		& \none		& \none		& \none		&*(green)3	& \none		& \none		& \none		&\none[3]	\\	
\end{ytableau}
\end{align*}
\end{example}

As with the other involutions, it is easy to see that $I_3$ reverses the sign and preserves the weight for all elements $(T,\tau)\in\mathcal{S}$ which are not fixed.  The elements of $\mathcal{S}$ which are fixed by $I_3$ are those which get fixed by $I_1$ in step (II) above.  Therefore, step (I) above defines a map $f:\mathcal{S}_\lambda^{I_3}\rightarrow\coprod\mathcal{T}_\sigma^{I_1}$ where the union is over all $\sigma$ which are obtained by adding a length $kn$ border strip to $\lambda$.  The map $f$ is clearly invertible, so $f$ is a bijection.  The function $f$ preserves the weight but does not quite preserve the sign.  In fact, $f$ introduces a factor of $-1$ for every shift in step (I) (corresponding to multiplying $\tau$ by a transposition).  This introduces a factor of $(-1)^{ht(\sigma\setminus\lambda)}$.  Putting it all together, we have
\begin{align*}
F_{\lambda,N}[n]&=\sum_{(T,\tau)\in\mathcal{S}_\lambda}(-1)^{\tau}x^T\\
&=\sum_{(T,\tau)\in\mathcal{S}_\lambda^{I_3}}(-1)^{\tau}x^T\\
&=\sum_\sigma(-1)^{ht(\sigma\setminus\lambda)}\sum_{(T',\tau')\in\mathcal{T}_\sigma^{I_1}}(-1)^{\tau}x^T\\
&=\sum_\sigma(-1)^{ht(\sigma\setminus\lambda)}x^\delta s_{\sigma,n}[N]
\end{align*}
where the last equality follows from Lemma \ref{lem1}.
\end{proof}

Lemmas \ref{lem2} and \ref{lem3} complete the proof of Theorem 1'.  Dividing both sides by $x^\delta$ and taking $N\rightarrow\infty$ proves Theorem 1.

\section{Proof of Theorem 2}\label{sec:theorem2}

In order to prove Theorem 2, define the degree of the variable $x_{i,j}$ to be $j$.  Then Theorem 2 follows from the next result by taking $N\rightarrow\infty$.

\begin{theorem2'}
The leading term of 
\[
\sum(-1)^{ht(\sigma\setminus\lambda)} s_{\sigma,N}^l[n]
\]
has degree bounded below by $N-kn-\frac{l}{n}N$.
\end{theorem2'}

\begin{proof}

We define the generating function $F_{\lambda,N}^l[n]$ exactly as we defined $F_{\lambda,N}[n]$ above, except we use the shifted weight defined in \eqref{shiftmon}.  Since the involution $I_3$ preserves the shifted weight (it only moves boxes along diagonals), Lemma \ref{lem3} carries through unchanged and proves that
\[
F_{\lambda,N}^l[n]=x^{\delta,l}\sum(-1)^{ht(\sigma\setminus\lambda)} s_{\sigma,N}^l[n]
\]
where $x^{\delta,l}$ is the shifted monomial associated to the \textit{standard} tableau on $\hat\emptyset$:
\[
\begin{ytableau}
*(green)1 	&*(white)1	&*(yellow)1	&*(green)1	&*(white)1	\\
\none 		&*(green)2	&*(white)2	&*(yellow)2	&*(green)2	\\		
\none		& \none		&*(green)3	&*(white)3	&*(yellow)3	\\
\none		& \none		& \none		&*(green)4	&*(white)4	\\
\none		& \none		& \none		& \none		&*(green)5	\\	
\end{ytableau}
\]

\begin{remark}\label{leadterm}
It is easy to see that $x^{\delta,l}$ has the smallest degree of any tableau on $\hat\emptyset$ which weakly increases along rows.
\end{remark}

Define $\mathcal{S}'\subset \mathcal{S}$ to be the subset of $\mathcal{S}$ consisting of tableaux of $\hat\lambda_i$ where the entries in the $i$th row do not exceed $N-kl$.  We define an involution $I_4$ on the elements of $\mathcal{S}'$ as follows.
\begin{enumerate}[(I)]
\item Remove the first $kn$ boxes from row $i$, slide the remaining boxes $kn$ units to the left and add $kl$ to each remaining entry.
\item If $m$ is the rightmost entry of the boxes which were removed in (I) ($m\leq N-kl$ by definition of $\mathcal{S}'$), subtract $kl$ from from each entry of row $j$ where $\tau_j=m+kl$, and then slide them to the right by $kn$ units and insert the boxes removed in (I).
\item Switch $\tau_i$ and $\tau_j$.
\end{enumerate}

Clearly $I_4$ is sign reversing and it preserves weight (this is the reason for adding/subtracting $kl$ to the entries when we slide them).  Therefore,
\[
F_{\lambda,N}^l[n]=\sum_{(T,\tau)\in\mathcal{S}\setminus\mathcal{S}'}(-1)^{\tau}x^{T,l}.
\]
But the rightmost entry of the $i$th row of every tableau in $\mathcal{S}\setminus\mathcal{S}'$ is at least $N-kl$ and this contributes at least $N-kl-\frac{l}{n}N$ to the degree of the associated monomial.  This implies that the degree of the associated monomial is at least $deg(x^{\delta,l})+N-kl-\frac{l}{n}N$.  Therefore, the degree of the leading term of $F_{\lambda,N}^l[n]$ (and hence $x^{\delta,l}\sum(-1)^{ht(\sigma\setminus\lambda)} s_{\sigma,N}^l[n]$) is at least $deg(x^{\delta,l})+N-kl-\frac{l}{n}N$.  Dividing by $x^{\delta,l}$ proves the theorem. 

\end{proof}

\bibliographystyle{alpha}

\end{document}